\documentclass{amsart}
\usepackage{amsfonts}
\usepackage{amsmath,amssymb}
\usepackage{amsthm}
\usepackage{amscd}
\usepackage{graphics}
\usepackage{graphicx}
\theoremstyle{remark}{
\newtheorem{Def}{{\rm Definition}}
\newtheorem{Ex}{{\rm Example}}

}
\newtheorem{Cor}{Corollary}
\newtheorem{Prop}{Proposition}
\newtheorem{Thm}{Theorem}

\newtheorem{Fact}{Fact}
\begin{document}
\title[Cohomology rings of Reeb spaces of new explicit fold maps]{New observations on cohomology rings of Reeb spaces of explicit fold maps}
\author{Naoki Kitazawa}
\address{Institute of Mathematics for Industry, Kyushu University, 744 Motooka, Nishi-ku Fukuoka 819-0395, Japan\\
 TEL (Office): +81-92-802-4402 \\
 FAX (Office): +81-92-802-4405}
\email{n-kitazawa@imi.kyushu-u.ac.jp}
\maketitle
\begin{abstract}
As a branch of algebraic topology and differential topology of manifolds, the theory of {\it Morse} functions and their higher dimensional versions or {\it fold} maps and its application to studies of manifolds is fundamental, important and interesting. This paper is on explicit construction of fold maps and homology groups and cohomology rings of their {\it Reeb spaces}: they are defined as the spaces of all connected components of preimages of the maps, and in suitable situations inherit some topological information such as homology groups and cohomology rings of the manifolds. 

Constructing these maps explicitly is a fundamental work and for example, will lead us to understand (important classes of) manifolds in more geometric and constructive ways: it seems to be a fundamental, important and new problem in geometry. The construction is also a difficult task even on a manifold which is not so complicated. The author constructed explicit fold maps systematically and performed several calculations of homology groups and cohomology rings of the Reeb spaces.
This paper concerns new observations on this task and presents explicit various Reeb spaces with their cohomology rings: they are of new types.
          

\end{abstract}


\maketitle
\section{Introduction and fundamental notation and terminologies.}
\label{sec:1}
  {\it Fold} maps are differentiable maps regarded as higher dimensional versions of Morse functions and fundamental and important
 tools in investigating algebraic topological and differential topological properties of manifolds. The present paper concerns explicit construction of fold maps. Constructing these maps on explicit manifolds is fundamental and important in these studies and it is also difficult. 
As a new important and interesting problem in geometry, we will develop this to understand the worlds (of some classes of differentiable) manifolds in geometric and constructive ways.
\cite{kitazawa5} and \cite{kitazawa6} are closely related to the present study and \cite{kitazawa}, \cite{kitazawa2}, \cite{kitazawa3}, and \cite{kitazawa4} are also closely related to the two papers before and the present study.
\subsection{Fold maps and their Reeb spaces}
\subsubsection{Fold maps}

First we review the definition of a {\it fold} map. Before this, we explain fundamental terminologies and notation related to {\it singular points} of differentiable maps. 

Throughout this paper, manifolds and maps between two manifolds are smooth (of class $C^{\infty}$) unless otherwise stated. A diffeomorphism is always smooth (of class $C^{\infty}$).

A {\it singular} point of a smooth map $c:X \rightarrow Y$ is a point $p$ at which the rank of the differential ${dc}_p$ of the map is smaller than $\dim Y$. A {\it singular value} of the map is a point realized as a value at a singular point. The set $S(c)$ of all singular points is the {\it singular set} of the map. The {\it singular value set} is the image $c(S(c))$ of the singular set. The {\it regular value set} of the map is the complement of the singular value set. A {\it regular value} is a point in the regular value set.

\begin{Def}
\label{def:1}
Let $m>n \geq 1$ be integers. A smooth map between an $m$-dimensional smooth manifold with no boundary and an $n$-dimensional smooth manifold with no boundary is said to be a {\it fold} map if at each singular point $p$, the function is represented as
$$(x_1, \cdots, x_m) \mapsto (x_1,\cdots,x_{n-1},\sum_{k=n}^{m-i}{x_k}^2-\sum_{k=m-i+1}^{m}{x_k}^2)$$
 for some coordinates and an integer $0 \leq i(p) \leq \frac{m-n+1}{2}$.
\end{Def}

For a fold map, the following three fundamental properties hold.
\begin{itemize}
\item For any singular point $p$, the $i(p)$ in Definition \ref{def:1} is unique {\rm (}$i(p)$ is called the {\it index} of $p${\rm )}. 
\item The set consisting of all singular points of a fixed index of the map is a smooth closed submanifold of dimension $n-1$ with no boundary of the $m$-dimensional manifold. 
\item The restriction to the singular set of the map is a smooth immersion.
\end{itemize}

We introduce terminologies on smooth immersions. For a smooth manifold $X$, we denote the tangent bundle by $TX$ and the tangent vector space at $p \in X$ by $T_p X \subset TX$.

For a smooth immersion $c:X \rightarrow Y$, a {\it crossing} is a point in $Y$ whose preimage consists of at least two points. A crossing $p \in Y$ is said to be {\it normal} if the following properties hold: ${dc}_q$ denotes the differential of $c$ at $q \in X$ as before.

\begin{enumerate}
\item $c^{-1}(p)$ is a finite set consisting of exactly $n(p)$ points.
\item $\dim {\bigcap}_{q \in c^{-1}(p)} {dc}_q(T_q X)+n(p)(\dim Y-\dim X)=\dim Y$.
\end{enumerate}

A {\it stable} fold map is a fold map 
such that the restriction to the singular set is a smooth immersion and that the crossings of the immersion are always normal. We should define a {\it stable} fold map via the notion of {\it {\rm (}Whitney $C^{\infty}${\rm )}} topologies of the spaces of smooth maps. However, the definition before is equivalent to this definition. For fundamental theory of fold maps and more general generic smooth maps including {\it stable} maps, see
 \cite{golubitskyguillemin} for example. In addition, in the present paper, essentially, for fold maps, we concentrate on studies of fold maps such that the restrictions to the singular sets are embeddings.

\subsubsection{Reeb spaces}
The {\it Reeb space} of a continuous map is defined as the space of all connected components of preimages of the map. 

 Let $X$ and $Y$ be topological spaces. For $p_1, p_2 \in X$ and for a continuous map $c:X \rightarrow Y$, 
 we define as $p_1 {\sim}_c p_2$ if and only if $p_1$ and $p_2$ are in
 a same connected component of $c^{-1}(p)$ for some $p \in Y$. 
Thus ${\sim}_{c}$ is an equivalence relation on $X$. We denote the quotient space $X/{\sim}_c$ by $W_c$.
\begin{Def}
\label{def:2}
 $W_c$ is the {\it Reeb space} of $c$.
\end{Def}

 We denote the induced quotient map from $X$ onto $W_c$ by $q_c$. We can define a continuous map $\bar{c}: W_c \rightarrow Y$ uniquely
 so that the relation $c=\bar{c} \circ q_c$ holds.

\begin{Prop}[\cite{shiota}]
\label{prop:1}
For (stable) fold maps, the Reeb spaces are polyhedra and the dimensions are equal to those of the target manifolds.
\end{Prop}

Reeb spaces are also fundamental and important tools in investigating manifolds via smooth maps whose codimensions are negative. For smooth maps of several suitable classes, they know topological properties of the manifolds much. We will present Proposition \ref{prop:2} as an example of such phenomena. 

A {\it simple} fold map $f$ is a fold map such that the restriction $q_f {\mid}_{S(f)}$ is injective. 

An {\it almost-sphere} is a homotopy sphere obtained by gluing two copies of a standard closed disc on the boundaries by a diffeomorphism. The class of almost-spheres accounts for the class of homotopy spheres except $4$-dimensional spheres which are not diffeomorphic to $S^4$: such manifolds are undiscovered now.

A {\it PID} means a so-called principal ideal domain. 

\begin{Prop}[\cite{saekisuzuoka}  (\cite{kitazawa2} and \cite{kitazawa3})]
\label{prop:2}
Let $m$ and $n$ be integers satisfying $m>n \geq 1$. Let $A$ be a commutative group. Let $M$ be a smooth, closed, connected and orientable manifold of dimension $m$ and $N$ be an $n$-dimensional smooth manifold with no boundary.
  
In this situation, for a simple fold map $f:M \rightarrow N$ such that preimages of regular values
 are always disjoint unions of almost-spheres and that indices of singular points are always $0$ or $1$, the following properties hold.

\begin{enumerate}
\item An induced homomorphism
 ${q_f}_{\ast}:{\pi}_j(M) \rightarrow {\pi}_j(W_f)$ is an isomorphism of groups and ${q_f}_{\ast}:H_j(M;A) \rightarrow H_j(W_f;A)$ and ${q_f}^{\ast}:H^j(W_f;A) \rightarrow H^j(M;A)$ are isomorphisms of modules over $A$ for $0 \leq j \leq m-n-1$.
\item Let $A$ be a commutative ring. Let $J$ be the set of all integers greater than or equal to $0$ and smaller than or equal to $m-n-1$. Let ${\oplus}_{j \in J} H^{j}(W_f;A)$ and ${\oplus}_{j \in J} H^{j}(M;A)$ be algebras where the sums and the products are canonically induced from the cohomology rings $H^{\ast}(W_f;A)$ and $H^{\ast}(M;A)$, respectively: if a product is of degree larger than $m-n-1$, then it is zero. In this situation, $q_f$ induces an isomorphism between these algebras
${\oplus}_{j \in J} H^j(W_f;A)$ and ${\oplus}_{j \in J} H^{j}(M;A)$ over $R$: this is given by the restriction of the original homomorphism ${q_f}^{\ast}:H^{\ast}(W_f;A) \rightarrow H^{\ast}(M;A)$. 
 \item
Let $A$ be a PID and the relation $m=2n$ hold. In this situation, the rank of the module $H_n(M;A)$ over $R$ is twice the rank of the module $H_n(W_f;A)$ over $R$. In
 addition, if $H_{n-1}(W_f;A)$ is a free module over $A$, then $H_n(W_f;A)$, $H_{n-1}(M;A)$ and $H_n(M;A)$ are free modules over $A$. 
\end{enumerate}
\end{Prop} 
See also \cite{reeb} for Reeb spaces, for example.
\subsection{Explicit fold maps and their Reeb spaces}
It is fundamental and important to construct explicit fold maps. However, even on manifolds which are not so complicated, it is difficult. We present known examples here.

For a topological space $X$, an {\it X-bundle} means a bundle whose fiber is $X$. For a smooth manifold $X$, a {\it smooth} $X$-bundle is an $X$-bundle whose structure group is (a subgroup of) the diffeomorphism group of $X$.
\begin{Ex}
\label{ex:1}
\begin{enumerate}
\item
\label{ex:1.1}
 A stable {\it special generic} map is a specific version of simple fold maps in Proposition \ref{prop:2} and defined as a stable fold map such that the index of each singular point is $0$.
Canonical projections of unit spheres are most simplest examples. According to \cite{saeki}, \cite{saeki2}, \cite{saekisakuma} and \cite{wrazidlo}, homotopy spheres which are not diffeomorphic to standard spheres do not admit special generic maps into Euclidean spaces if the dimensions of the target Euclidean spaces are sufficiently high and smaller than the dimensions of the homotopy spheres.  

Furthermore, the maximal degree $j=m-n-1$ in Proposition \ref{prop:2} can be replaced by $j=m-n$ for a special generic map.

Last, the Reeb space of a (stable) special generic map from a connected and closed manifold of dimension $m$ into ${\mathbb{R}}^n$ satisfying the relation $m>n \geq 1$ is an $n$-dimensional connected and compact manifold we can immerse into ${\mathbb{R}}^n$. This is a fundamental fact explained in \cite{saeki} and so on. We will explain more in Fact \ref{fact:1}.
\item
\label{ex:1.2}
 (\cite{kitazawa}, \cite{kitazawa2} and \cite{kitazawa4})
Let $m>n \geq 1$ be integers. We can construct a stable fold map on a manifold represented as a connected sum of $l>0$ total spaces of smooth $S^{m-n}$-bundles over $S^n$ into ${\mathbb{R}}^n$ satisfying the following three properties.
\begin{enumerate}
\item The singular value set is embedded concentric spheres and by a suitable diffeomorphism on ${\mathbb{R}}^n$, we can map the set to ${\sqcup}_{k=1}^{l+1} \{||x||=k \mid x \in {\mathbb{R}}^n\}$.
\item Preimages of regular values are disjoint unions of standard spheres and n the target space ${\mathbb{R}}^n$, the number of connected components of a preimage increases as we go straight to the origin $0 \in {\mathbb{R}}^n$ of the target Euclidean space starting from a point in the complement of the image. 
\end{enumerate}
The Reeb space is simple homotopy equivalent to a bouquet of $l$ copies of a sphere of dimension $n$.
\item
\label{ex:1.3}
 In \cite{kitazawa5} and \cite{kitazawa6}, stable fold maps such that the restrictions to the singular sets are embedding satisfying the assumption of Proposition \ref{prop:2} have been constructed. They have been obtained by finite iterations of surgery operations ({\it bubbling operations}) starting from fundamental fold maps. More precisely, starting from stable special generic maps and so on, by changing maps and manifolds by bearing new connected components of singular sets one after another, we obtain desired maps.
\end{enumerate}
\end{Ex}

\subsection{Construction of explicit fold maps by bubbling operations and the organization of this paper.}

In this paper, we present further studies on construction in Example \ref{ex:1} (\ref{ex:1.3}). We use bubbling operations. We try new methods and obtain new families of explicit stable fold maps.

The organization of the paper is as the following.
In the next section, we introduce {\it bubbling operations} first introduced in \cite{kitazawa}.
The last section is devoted to main results. We present construction of new families of explicit fold maps, investigate the cohomology rings of the Reeb spaces and we can observe new explicit cases and phenomena. Proposition \ref{prop:2} is a key tool in knowing information on the cohomology rings of the manifolds from Reeb spaces in suitable cases.

Throughout this paper, $M$ is a smooth, closed and connected manifold of dimension $m$, $N$ is a smooth manifold of dimension $n$ with no boundary, the relation $m>n \geq 1$ holds and $f:M \rightarrow N$ is a smooth map. In addition, the structure groups of bundles such that the fibers are manifolds are assumed to be
 (subgroups of) the diffeomorphism groups or equivalently, the bundles are smooth. A {\it linear} bundle is a smooth bundle whose fiber is a $k$-dimensional unit disc or the ($k-1$)-dimensional unit sphere in ${\mathbb{R}}^{k}$ and whose structure group is a subgroup of the $k$-dimensional orthogonal group $O(k)$ acting in a canonical way for $k \geq 1$. 

\thanks{The author is a member of and supported by the project Grant-in-Aid for Scientific Research (S) (17H06128 Principal Investigator: Osamu Saeki)
"Innovative research of geometric topology and singularities of differentiable mappings"
( 
https://kaken.nii.ac.jp/en/grant/KAKENHI-PROJECT-17H06128/
).}


\section{Bubbling operations to fold maps.}
\label{sec:2}



We introduce {\it bubbling operations}, first introduced in \cite{kitazawa5}, referring to the article. We revise some terminologies and ideas from the original definition. However, essentially most of notions on the operation are same. The explanation overlaps some explanation in \cite{kitazawa6}.
\begin{Def}
\label{def:3}
For a stable fold map $f:M \rightarrow N$, let $P$ be a connected
 component of $(W_f-q_f(S(f))) \bigcap {\bar{f}}^{-1}(N-f(S(f)))$, which we may regard as an open manifold diffeomorphic to
 an open manifold $\bar{f}(P)$ in $N$. Let $S$ be a connected, orientable and compact submanifold with no boundary of
 $P$ such that the normal bundle is orientable. Let $N(S)$, ${N(S)}_i$ and ${N(S)}_o$ be small closed tubular neighborhoods
 of $S$ in $P$ such that relations ${N(S)}_i \subset {\rm Int }N(S)$ and $N(S) \subset {\rm Int} {N(S)}_o$ hold. 
Let $Q:={q_f}^{-1}({N(S)}_o)$: $q_f {\mid}_{Q}$
 makes $Q$ a smooth bundle over ${N(S)}_o$.
Assume that we can construct a stable fold map $f^{\prime}$ of an $m$-dimensional closed manifold $M^{\prime}$ into ${\mathbb{R}}^n$ satisfying the following properties.
\begin{enumerate}
\item $M-{\rm Int} Q$ is realized as a compact submanifold (with non-empty boundary) of $M^{\prime}$ of dimension $m$ via a suitable smooth embedding $e:M-{\rm Int} Q \rightarrow M^{\prime}$.
\item $f {\mid}_{M-{\rm Int} Q}={f}^{\prime} \circ e {\mid}_{M-{\rm Int} Q}$ holds.
\item ${f}^{\prime}(S({f}^{\prime}))$ is the disjoint union of $f(S(f))$ and $\bar{f}(\partial N(S))$.
\item $(M^{\prime}-e(M-Q)) \bigcap {q_{f^{\prime}}}^{-1}({N(S)}_i)$ is empty or ${{q_f}^{\prime}} {\mid}_{(M^{\prime}-e(M-Q)) \bigcap {q_{f^{\prime}}}^{-1}({N(S)}_i)}$ makes $(M^{\prime}-e(M-Q)) \bigcap {q_{f^{\prime}}}^{-1}({N(S)}_i)$ a bundle over ${N(S)}_i$.
\end{enumerate}
This situation yields a procedure of constructing $f^{\prime}$ from $f$. We call it a {\it normal bubbling operation} to $f$. Furthermore, $S$ is called the {\it generating manifold} of the operation. 

Furthermore, we can define the following two.
\begin{enumerate}
\item
 ${{f}^{\prime}} {\mid}_{(M^{\prime}-e(M-Q)) \bigcap {q_{f^{\prime}}}^{-1}({N(S)}_i)}$ makes $(M^{\prime}-e(M-Q)) \bigcap {q_{f^{\prime}}}^{-1}({N(S)}_i)$ the disjoint union of two bundles over $N(S)$. In this case, the procedure is called a {\it normal M-bubbling operation} to $f$.
\item ${{f}^{\prime}} {\mid}_{(M^{\prime}-e(M-Q)) \bigcap {q_{f^{\prime}}}^{-1}({N(S)}_i)}$ makes $(M^{\prime}-e(M-Q)) \bigcap {q_{f^{\prime}}}^{-1}({N(S)}_i)$ the disjoint union of two bundles over ${N(S)}_i$ and the fiber of one of the bundles is an almost-sphere. In this case, the procedure is called a {\it normal S-bubbling operation} to $f$.
\end{enumerate}

In the definitions above, we call the fiber of a bundle given by ${q_{{f}^{\prime}}} {\mid}_{(M^{\prime}-e(M-Q)) \bigcap {q_{f^{\prime}}}^{-1}({N(S)}_i)}$ the preimage {\it born by a bubble} if the surgery is not an M-bubbling operation or an S-bubbling operation. In the case of an M-bubbling operation or an S-bubbling operation, connected components of the fiber are also called preimages {\it born by a bubble}. 

In the present situation, let $S$ be represented as the bouquet of finitely many connected, orientable and compact submanifolds without boundaries whose dimensions are smaller than $n$ and let the normal bundles of the submanifolds be orientable
Furthermore, let $N(S)$, ${N(S)}_i$ and ${N(S)}_o$ be small regular neighborhoods
 of $S$ in $P$ such that relations ${N(S)}_i \subset {\rm Int } N(S)$ and $N(S) \subset {\rm Int} {N(S)}_o$
 hold and that these three are isotopic as regular neighborhoods. Similarly,
 we can define similar operations and call the operations a {\it bubbling operation}, an {\it M}{\rm (}{\it S}{\rm )}{\it-bubbling operation} to $f$, and so on. We
 call $S$ the {\it generating polyhedron} of the operation.

\end{Def} 

Hereafter, we abuse notation in Definition \ref{def:3}.

\begin{Ex}
\label{ex:2}
In Definition \ref{def:3}, if $S$ is in an open disc of $P$ or more generally, $q_f {\mid}_{Q}$
 makes $Q$ a trivial bundle over ${N(S)}_o$, then we can perform a (normal) bubbling operation to $f$. 
We can replace the phrase "bubbling" by "M-bubbling" and "S-bubbling". 

Furthermore, we can change the fiber of the original bundle so that the resulting new connected components of the preimage containing no singular point or each connected component of the preimage born by a bubble is a manifold in the following explanations.
\begin{enumerate}
\item Any closed and connected manifold obtained by a handle attachment: more precisely consider the product of the original manifold appearing as the original fiber and the closed interval, attach a handle to one of the connected components of the boundary of the cylinder and consider the resulting connected component of the resulting boundary.
\item A disjoint union of arbitrary two manifolds whose connected sum is diffeomorphic to the original manifold appearing as the fiber.
\end{enumerate}
\end{Ex}

Ideas for the operations are based on stuffs in \cite{kobayashi}, \cite{kobayashi2} and \cite{kobayashisaeki}. Especially, through \cite{kobayashi} and \cite{kobayashi2}, {\it bubbling surgeries} have been established: a {\it bubbling surgery} accounts for the case where the generating polyhedron is a point.  

Hereafter, we introduce key notions, tools and propositions. They are also explained in \cite{kitazawa5} and \cite{kitazawa6}. For example, definitions may be different from ones in these articles. However, essentially they are same.
\begin{Cor}
\label{cor:1}
Let $f$ be a stable fold map. If an M-bubbling operation is performed to $f$ and a new map $f^{\prime}$ is obtained, then $W_f$ is a proper subset of $W_{{f}^{\prime}}$ and for the map $\bar{{f}^{\prime}}:W_{f^{\prime}} \rightarrow N$, the restriction to $W_f$ is $\bar{f}:W_f \rightarrow N$.
\end{Cor}
We apply this in several situations implicitly in the present paper.

The following proposition is a fundamental and
 key tool.
\begin{Prop}
\label{prop:3}
Let $f:M \rightarrow N$ be a stable fold map. Let $f^{\prime}:{M}^{\prime} \rightarrow N$ be a fold map obtained by an
 M-bubbling operation to $f$. Let $S$ be the generating polyhedron of this M-bubbling
 operation. Let $k$ be a positive integer and $S$ be represented as the bouquet of submanifolds $S_j$ with no boundaries where $j$ is an integer satisfying $1 \leq j \leq k$.  
In this situation, for each integer $0 \leq i<n$, there exist an isomorphism 
$$H_{i}(W_{{f}^{\prime}};R) \cong H_{i}(W_f;R) \oplus {\oplus}_{j=1}^{k} (H_{i-(n-{\dim} S_j)}(S_j;R))$$
 
and an isomorphism $H_{n}(W_{{f}^{\prime}};R) \cong H_{n}(W_f;R) \oplus R$.
\end{Prop}
The following proof is essentially same as a proof of the same proposition in \cite{kitazawa6}.
Rigorous proofs with discussions on Mayer-Vietoris sequences, homology groups of product
 bundles, and so on, are presented in \cite{kitazawa5}.

\begin{proof}
For $S_j$, we can take a small closed tubular neighborhood, regarded as
 the total space of a linear $D^{n-\dim S_j}$-bundle over $S_j$. By the definition of the operation, a small regular neighborhood $N(S)$ of $S$ is represented as a boundary connected sum of these
 closed tubular neighborhoods (isotoped slightly in a suitable way if we need). $W_{f^{\prime}}$ is obtained by attaching a manifold $B(S)$ represented as a connected sum of total spaces
 of linear $S^{n-\dim S_j}$-bundles over $S_j$ ($1 \leq j \leq k$) by considering $D^{n-\dim S_j}$ in the beginning as a hemisphere of $S^{n-\dim S_j}$ and identifying the subspace obtained by restricting the
 space to fibers $D^{m-\dim S_j}$ with the original regular neighborhood. For the manifold represented as a connected sum of total spaces of linear $S^{n-\dim S_j}$-bundles
 over $S_j$, the bundles are orientable and admit sections, corresponding to the submanifolds $S_j$ and regarded as subbundles whose fibers are one point sets $\{0\} \subset D^{n-\dim S_j} \subset S^{n-\dim S_j}$. Careful observations of $W_f$ and $W_{f^{\prime}}$ yield the result.  
\end{proof}
More precisely, a class in $H_{i-(n-{\dim} S_j)}(S_j;R)$ in the isomorphism is, by a natural isomorphism, mapped to an element in $H_{i}(W_{{f}^{\prime}};R)$ considering a cycle of $S_j \subset S \subset W_f \subset W_{f^{\prime}}$ representing the class and the class represented by a fiber diffeomorphic to $S^{n-\dim S_j}$ of the bundle. 

\begin{Def}
\label{def:4}
In the situation of Proposition \ref{prop:3}, as above, for each ($i-(n-\dim S_j)$)-cycle $c$ of an ingredient $S_j$ in the generating polyhedron $S \subset W_f$ and the class $[c] \in H_{i-(n-{\dim} S_j)}(S_j;R)$ represented by this, by a suitable homomorphism, we can define the class ${\rm bub}_{(f,f^{\prime}),S}(c) \in H_i(W_{f^{\prime}};R)$ represented by the cycle just before. We call this class the {\it bubbled element} of $c$.
\end{Def}

\begin{Def}
\label{def:5}
In the situation of Proposition \ref{prop:3},
$W_{f^{\prime}}$ is obtained by attaching a manifold $B((f,f^{\prime}),S)$ represented as a connected sum of total spaces
 of linear $S^{n-\dim S_j}$-bundles over $S_j$ ($1 \leq j \leq k$) by considering $D^{n-\dim S_j}$ in the beginning as a hemisphere of $S^{n-\dim S_j}$ and identifying the subspace obtained canonically by restricting the
 space to fibers $D^{m-\dim S_j}$ with the original regular neighborhood as explained in the proof. Note that the subspace is represented as a boundary connected sum of total spaces
 of linear $D^{n-\dim S_j}$-bundles over $S_j$ ($1 \leq j \leq k$). We call $B((f,f^{\prime}),S)$ a {\it bubbled space} of the generating polyhedron $S$.
\end{Def}

For a finite iteration of $M$-bubbling operations to a stable fold map $f$ to obtain a fold map $f^{\prime}$, we can define {\it bubbled elements} of classes in (each ingredient of) each generating polyhedron, {\it bubbled spaces} of the generating polyhedra, and so on, similarly. See also \cite{kitazawa6}.

\section{Main results}
\label{sec:3}
We will investigate homology groups and cohomology rings of the resulting Reeb spaces
 and present these results as main results.

For a graded commutative module over a commutative ring, we set the {\it $i$-th module} as the module consisting of all elements of degree $i$.
For a graded commutative algebra over a commutative ring in the present paper, we set the $0$-th module as the ring forgetting the ring structure and defining the action by the ring in the canonical way.

\subsection{CPS manifolds}
In \cite{kitazawa6}, a {\it CPS} manifold and a {\it CPS} and {\it GCPS} graded commutative algebras are defined.
\begin{Def}
A manifold $S$ is said to be {\it CPS} if either of the following two properties hold.
\begin{enumerate}
\item $S$ is a standard sphere whose dimension is positive.
\item $S$ is represented as a connected sum or a product of two CPS manifolds.
\end{enumerate}
\end{Def}
We can know the following facts by virtue of fundamental differential topological discussions and omit the proof.
\begin{Prop}
\label{prop:5}
\begin{enumerate}
\item For an integer $k>0$, $k$-dimensional CPS manifolds can be embedded into ${\mathbb{R}}^{k+1}$.
\item All CPS manifolds admit orientation reversing diffeomorphisms.
\end{enumerate}
\end{Prop}

\begin{Def}
A graded commutative algebra $A$ over a PID $R$ is said to be {\it CPS} if 
it is isomorphic to the cohomology ring of some CPS manifold.
\end{Def}
\begin{Def}
A graded commutative algebra $A$ over a PID $R$ is said to be {\it GCPS} if it is isomorphic to the cohomology ring of a point or that of a bouquet of a finite number of CPS manifolds whose coefficient ring  $R$. 
\end{Def}

\subsection{A connected sum of two smooth maps whose codimensions are negative.}
We introduce a {\it connected sum} of two smooth maps whose codimensions are negative. We give an essentially same presentation
 in \cite{kitazawa6}.

Let $m>n \geq 1$ be integers, $M_i$ ($i=1,2$) be a closed and connected manifold of dimension $m$ and $f_1:M_1 \rightarrow {{\mathbb{R}}}^n$
and $f_2:M_2 \rightarrow {\mathbb{R}}^n$ be smooth maps.

Set ${{\mathbb{R}}^n}^+:=\{x=(x_1.\cdots,x_n) \in {\mathbb{R}}^n \mid x_1 \geq 0 \}$. 
Let $P_i \subset {\mathbb{R}}^n$ ($i=1,2$) be the set such that for a diffeomorphism ${\phi}_i$ on ${\mathbb{R}}^n$, $P_i={\phi}_i({{\mathbb{R}}^n}^+)$ holds.
 We assume that for the map $f_i {\mid}_{{f_i}^{-1}(P_i)}:{f_i}^{-1}(P_i) \rightarrow P_i$, there exists a pair $(\Phi,\phi)$ of diffeomorphisms satisfying the relation $\phi \circ f_1 {\mid}_{{f_1}^{-1}(P_1)}=f_2 {\mid}_{{f_2}^{-1}(P_2)} \circ \Phi$. We consider the canonical projection of a unit sphere $S^m \subset {\mathbb{R}}^{m+1}$ to ${\mathbb{R}}^n$ defined as the composition of the canonical inclusion with the projection ${\pi}_{m+1,n}((x_1, \cdots, x_n, \cdots, x_{m+1}))=(x_1, \cdots, x_n)$ and the restriction to the intersection of ${{\mathbb{R}}^{m+1}}^{+}:=\{x=(x_1,\cdots,x_{m+1}) \in {\mathbb{R}}^{m+1} \mid x_1 \geq 0\}$ and $S^m$: we also restrict the target space to ${{\mathbb{R}}^n}^+:=\{x=(x_1.\cdots,x_n) \in {\mathbb{R}}^n \mid x_1 \geq 0 \}$. We denote this map by ${\pi}_{m,n,S,+}$. We also assume that there exists a pair $({\Phi}_i,{\phi}_i)$ of diffeomorphisms $\Phi$ and $\phi$ satisfying the relation ${\phi}_i \circ f_i {\mid}_{{f_i}^{-1}(P_i)}={\pi}_{m,n,S,+} \circ {\Phi}_i$ for $i=1,2$.


We can glue the maps ${f_i} {\mid}_{{f_i}^{-1}({\mathbb{R}}^n-{\rm Int} P_i)}:{f_i}^{-1}({\mathbb{R}}^n-{\rm Int} P_i) \rightarrow {\mathbb{R}}^n-{\rm Int} P_i$ ($i=1,2$) on the boundaries to
 obtain a new map and by composing a diffeomorphism from the new target space to ${\mathbb{R}}^n$, we obtain a smooth map into ${\mathbb{R}}^n$ so that the resulting manifold is represented as a connected sum of the original two manifolds. 
The resulting map is called a {\it connected sum} of $f_1$ and $f_2$.
For arbitrary two stable fold maps on closed and connected manifolds of dimension $m$ into ${\mathbb{R}}^n$, we can obtain a connected sum of them. This procedure is, as explained in \cite{kitazawa6} for example, a fundamental tool in constructing a new stable fold map of a given class on a manifold represented as a connected sum of two manifolds admitting maps in the class.
 
\subsection{Special generic maps into Euclidean spaces}
The following fact is a fundamental fact on special generic maps stated in \cite{saeki} and also in Example \ref{ex:1} (\ref{ex:1.2}).
\begin{Fact}
\label{fact:1}
Let $m>n \geq 1$ be integers.
\begin{enumerate}
\item For a closed and connected manifold of dimension $m$, it admits a special generic map into ${\mathbb{R}}^n$ if and only if it admits a stable simple fold map as in Proposition \ref{prop:3} whose Reeb space is a compact and connected manifold of dimension $n$ we can immerse into ${\mathbb{R}}^n$.
\item For a compact and connected manifold of dimension $n$ we can smoothly immerse into ${\mathbb{R}}^n$, there exists a closed and connected manifold of dimension $m$ admitting a stable special generic map into ${\mathbb{R}}^n$ whose Reeb space is diffeomorphic to the $n$-dimensional compact manifold. Furthermore, the special generic map satisfies the following properties. 
\begin{enumerate}
\item On the preimage of the interior of the Reeb space, the map onto this interior gives a $S^{m-n}$-bundle. 
\item On the preimage of a small collar neighborhood of the Reeb space, the composition of the map onto this neighborhood with the canonical projection onto the boundary gives a linear $S^{m-n+1}$-bundle.
\end{enumerate}
\end{enumerate}
\end{Fact}
\begin{Def}
In the latter half of Fact \ref{fact:1} above, we can construct the map so that the two bundles are trivial smooth bundles (which may not be trivial linear bundles) on a suitable $m$-dimensional manifold. In this case, if the restriction to the singular set is an embedding, then we say that the map has {\it almost-trivial monodromies}.
\end{Def}
In Example \ref{ex:1} (\ref{ex:1.1}) special generic maps whose Reeb spaces are standard discs are presented, for example. They have almost-trivial monodromies.

Including the examples, most of special generic maps in existing studies essentially satisfy the following condition.
\begin{Def}
Let $m>n\geq 1$ integers.
A special generic map from a closed and connected manifold of dimension $m$ into ${\mathbb{R}}^n$ is said to be a {\it standard GCPS} special generic map if the following properties hold.
\begin{enumerate}
\item The Reeb space is represented as a boundary connected sum of finitely many manifolds each of which is represented as a product of a CPS manifold and a standard closed disc: it is simple homotopy equivalent to a bouquet of finite numbers of CPS manifolds.
\item The restriction map to the singular set is an embedding.
\end{enumerate}
\end{Def}
\begin{Ex}
\label{ex:3}
\begin{enumerate}
\item Let $n=2$. A result in \cite{saeki} shows that a stable special generic map into the plane on a manifold of dimension $m \geq 3$ is essentially regarded as a standard GCPS special generic map. In fact the Reeb space is obtained as a boundary connected sum of finite copies of $S^1 \times D^1$.
\item
Let $n=3,4$.
Saeki's results and Nishioka's results, obtained later (\cite{nishioka} and \cite{saeki}), imply that a $5$-dimensional closed and simply-connected manifold admits a special generic map into ${\mathbb{R}}^n$ if and only if it is represented as a connected sum of the total spaces of $S^3$-bundles over $S^2$. According to the results, such a manifold also admits a standard GCPS special generic map into ${\mathbb{R}}^n$ whose Reeb space is represented as a boundary connected sum of finite copies of $S^2 \times D^{n-2}$.   
\end{enumerate}
\end{Ex}
\begin{Def}
Let $m>n \geq 1$ be integers.
Consider a stable fold map $f:M \rightarrow N$ on a closed and connected manifold of dimension $m$ into a manifold with no boundary of dimension $n$ such that the restriction to the singular set is an embedding. We can see the following properties.

\begin{enumerate}
\item On the preimage of each connected component of $W_f-q_f(S(f))$, the map onto this interior gives a trivial bundle. 
\item On the preimage of a small regular neighborhood of each connected component of $q_f(S(f))$, the composition of the map onto this neighborhood with the canonical projection onto $q_f(S(f))$ gives a smooth bundle.
\end{enumerate}

If these bundles are smooth trivial bundles, then we say that the map has {\it almost-trivial monodromies}.
\end{Def}

\begin{Ex}
\label{ex:1.4}
Example \ref{ex:2} implies that we can construct stable fold maps having almost-trivial monodromies one after another by considering compact and orientable submanifolds with no boundaries in the Reeb spaces preimages of which contain no singular points starting from a stable fold map having almost-trivial monodromies. 
\end{Ex}

\subsection{New results}
We show new results on cohomology rings of Reeb spaces. In the situation where we can apply Proposition \ref{prop:2}, we can obtain a result on those of the resulting manifolds. For example, starting from a stable standard special generic map and performing a finite iteration of S-bubbling operations starting from the map, we can obtain various maps to which we can apply Proposition \ref{prop:2}.

We introduce the {\it dual} of a suitable element of an module. 
Let $R$ be a commutative ring having exactly one identity element $1 \neq 0 \in R$. 
Let $A$ be an module over $R$.
Let $a \in A$ be a non-zero element satisfying the following properties.
\begin{enumerate}
\item We cannot represent as $a=ra^{\prime}$ for some $r \in R$ which is not a unit and $a^{\prime} \in A$.
\item $ra = 0$ if and only if $r=0 \in R$.
\item Let $<a>$ be the submodule of $A$ generated by the one element set $\{a\} \subset A$ and $A$ is represented as the internal direct sum of $<a>$ and a suitable submodule $B \subset A$. 
\end{enumerate}
\begin{Def}
We call $a$ as above a {\it unit free generator} or {\it UFG} of $A$. 
\end{Def}
We can define a homomorphism $a^{\ast}$ from $A$ into $R$ between graded commutative algebras over $R$ satisfying the following properties.
\begin{enumerate}
\item $a^{\ast}(a)=1 \in R$.
\item $a^{\ast}(b)=0$ for $b \in B$ and any $B$ as in the condition.
\end{enumerate}
\begin{Def}
We call $a^{\ast}$ the dual of $a$.
\end{Def}
More explicitly, let $A$ be a homology group whose coefficient ring is a PID $R$. In this case, the dual of an element $a \in A$ is canonically regarded as a cohomology class. We abuse this principle in this section.
\begin{Thm}
\label{thm:1}

Let $R$ be a PID.
Let $f:M \rightarrow {\mathbb{R}}^n$ be a stable fold map. We consider
 a special generic map $f_0$ such that the following properties hold.
\begin{enumerate}
\item The restriction $f_0 {\mid}_{S(f_0)}$ is an embedding.
\item The Reeb space $W_{f_0}$ is a compact and connected manifold we can smoothly embed into ${\mathbb{R}}^n$.
\end{enumerate}
We also consider a finite iteration of normal bubbling operations to the map $f_0$ to obtain a new map $f_1$. Next, we consider a connected sum of $f$ and the map $f_1$ just before. The following two statements hold. 
\begin{enumerate}
\item \label{thm:1.1}
By this construction, we can obtain a map $f^{\prime}$ satisfying the following properties.

\begin{enumerate}
\item 
\label{thm:1.1.1}
$H^{\ast}(W_{f^{\prime}};R)$ is isomorphic to and identified with a graded commutative algebra $B$ over $R$ obtained from the direct sum of $H^{\ast}(W_{f};R)$ and a suitable algebra $A$, which we will explain later. More precisely, $B$ satisfies the following properties.
\begin{enumerate}
\item The $i$-th module is the direct sum of the $i$-th modules of the two summands for $i>0$.
\item For a pair $(a_{i_1,1},a_{i_1,2}) \in H^{\ast}(W_{f};R) \oplus A$ of elements of degree $i_1>0$ and a pair $(a_{i_2,1},a_{i_2,2}) \in H^{\ast}(W_{f};R) \oplus A$ of elements of degree $i_2>0$, which are defined as elements of degree $i_1$ and degree $i_2$, respectively, the product is $(a_{i_1,1}a_{i_2,1},a_{i_1,2}a_{i_2,2})$ and of degree $i_1+i_2$.
\item For $r \in R$, which is also an element of degree $0$ and a pair $(a_{i,1},a_{i,2}) \in A_1 \oplus A_2$ of elements of degree $i>0$, which is defined as an element of degree $i$, the product is $(ra_{i,1},ra_{i,2}) \in A_1 \oplus A_2$ and of degree $i$.
\end{enumerate} 

\item
\label{thm:1.1.2}
$A$ is, as a graded module over $R$, represented as a direct sum of a free finitely generated module over $R$ obtained by forgetting the ring structure of a graded commutative algebra $A_0$ over $R$, which is isomorphic to and identified with the cohomology ring of a compact and connected manifold of dimension $n$ we can embed smoothly into ${\mathbb{R}}^n$ whose coefficient ring is $R$, and a finitely generated graded module over $R$.
\item 
\label{thm:1.1.3}
$A_0$ is regarded as a subalgebra of $A$ in a canonical way.
\end{enumerate}
\item
\label{thm:1.2}
 Let $n \geq 3$. Let $l \geq 0$ be an integer. Let $\{S_j\}_{j=1}^{l}$ be a family of closed, connected and orientable manifolds satisfying the following properties.
\begin{itemize}
\item $2{\dim} S_j \leq n$.
\item There exists $c_{S_j} \in H_{k_j}(S_j;R)$ for some $k_j \geq 0$ satisfying the following properties.
\begin{itemize}
\item $c_{S_j}$ is a UFG of $H_{k_j}(S_j;R)$. 
\item $c_{S_j}$ is represented by a closed submanifold $F_j \subset S_j$ with no boundary whose normal bundle is trivial.
\end{itemize}
\end{itemize}
Suppose also that $f_0$ has almost-trivial monodromies and the generating polyhedra for construction of $f_1$ are of dimension smaller than $n-1$. In this situation, we can obtain a map $f^{\prime \prime}$ satisfying the following properties by a finite iteration of M{\rm (}S{\rm )}-operations to $f^{\prime}$ before such that at each step the generating manifold is diffeomorphic to $S_j$ for $1 \leq j \leq l$: we use this notation also for these generating manifolds.
\begin{enumerate}
\item 
\label{thm:1.2.1}
The module $H_{i}(W_{f^{\prime \prime}};R)$ over $R$ is represented as a direct sum of $H_{i}(W_{f^{\prime}};R)$ and ${\oplus}_{j=1}^{l} H_{i-(n-\dim S_j)}(S_j;R)$. $H^{i}(W_{f^{\prime \prime}};R)$ is as an module over $R$ represented as a direct sum of $H^{i}(W_{f^{\prime}};R)$ and ${\oplus}_{j=1}^{l} H^{i-(n-\dim S_j)}(S_j;R)$. Moreover, we can define the {\it bubbled element} ${\rm bub}_{(f^{\prime},f^{\prime \prime}),S_j}(c) \in H^{i}(W_{f^{\prime \prime}};R)$ of $c \in H^{i-(n-\dim S_j)}(S_j;R)$ as is done in Proposition \ref{prop:3} satisfying similar properties. Let $i_{(f^{\prime},f^{\prime \prime})}:W_{f^{\prime}} \rightarrow W_{f^{\prime \prime}}$ be the canonical inclusion.

In this situation, we can define a monomorphism ${i_{(f^{\prime},f^{\prime \prime})}}^{{\ast}^{\prime}}:H^{j}(W_{f^{\prime}};R) \rightarrow H^{j}(W_{f^{\prime \prime}};R)$ and an isomorphism ${\phi}_{(f^{\prime},f^{\prime \prime})}:H^{i}(W_{f^{\prime}};R) \oplus {\oplus}_{j=1}^{l} (H^{i-(n-\dim S_j)}(S_j;R)) \rightarrow H^{i}(W_{f^{\prime \prime}};R)$ by ${\phi}_{(f^{\prime},f^{\prime \prime})}((c_1,c_2))={i_{(f^{\prime},f^{\prime \prime})}}^{{\ast}^{\prime}}(c_1)+{\rm bub}_{(f^{\prime},f^{\prime \prime}),S_j}(c_2)$
 for $c_1 \in H^{i}(W_{f^{\prime}};R)$ and $c_2 \in H^{i-(n-\dim S_j)}(S_j;R)$.
Moreover, corresponding the duals to bubbled elements which are UFGs, we can obtain an isomorphism from a free submodule of ${\rm bub}_{(f^{\prime},f^{\prime \prime}),S_j}(H_{i-(n-\dim S_j)}(S_j;R))$ onto a free submodule of ${\rm bub}_{(f^{\prime},f^{\prime \prime}),S_j}(H^{i-(n-\dim S_j)}(S_j;R))$ and if $H_{i-(n-\dim S_j)}(S_j;R)$ is free, then this isomorphism can be obtained as one from ${\rm bub}_{(f^{\prime},f^{\prime \prime}),S_j}(H_{i-(n-\dim S_j)}(S_j;R))$ onto the submodule of ${\rm bub}_{(f^{\prime},f^{\prime \prime}),S_j}(H^{i-(n-\dim S_j)}(S_j;R))$ generated by the set of all UFGs. 
\item
\label{thm:1.2.2}
The products of two elements in $H^{\ast}(W_{f^{\prime \prime}};R)$ satisfy the following properties.
\begin{enumerate}
\item
\label{thm:1.2.2.1}
 Products of elements in ${i_{(f^{\prime},f^{\prime \prime}),S}}^{{\ast}^{\prime}}(H^{\ast}(W_{f^{\prime}};R))$ are canonically induced ones.
\item
\label{thm:1.2.2.2}
 The product of an element in ${\rm bub}_{(f^{\prime},f^{\prime \prime}),S_{j_1}}(H^{i_1-(n-\dim S_{j_1})}(S_{j_1};R))$ and an element in ${\rm bub}_{(f^{\prime},f^{\prime \prime}),S_{j_2}}(H^{i_2-(n-\dim S_{j_2})}(S_{j_2};R))$ always vanishes for $i=i_1,i_2 \geq \frac{n}{2}$ and $1 \leq j_1,j_2 \leq l$.
\item 
\label{thm:1.2.2.3}
Set $S:=S_j$, $F:=F_j$ and $k:=k_j$ for $1 \leq j \leq l$. 
Define $Q_{\dim S-k,W_{f_0},R} \subset H_{\dim S-k}(W_{f_0};R)$ as the submodule of satisfying the following properties.
\begin{enumerate}
\item It is generated by UFGs represented by standard spheres smoothly embedded in the interior of a small colloar neighborhood of $\partial W_{f_0}$.
\item It is free.
\item $H_{\dim S-k}(W_{f_0};R)$ is the internal direct sum of a submodule ${\bar{Q}}_{\dim S-k,W_{f_0},R}$ and it.                                                                                                    
\end{enumerate}
Let the rank of $Q_{\dim S-k,W_{f_0},R} \subset A_0$ be $b \geq 0$ and take a basis $\{e_{j^{\prime}}\}_{j^{\prime}=1}^{b}$. Note that $A_0$ is regarded as an subalgebra $H^{\ast}(W_{f_0};R) \subset H^{\ast}(W_{f^{\prime}};R)$, free, and also that as a graded module over $R$, $Q_{\dim S-k,W_{f_0},R}$ is isomorphic to and identified with a submodule of $A_0$ via the homomorphism on $Q_{\dim S-k,W_{f_0},R} \subset H_{\dim S-k}(W_{f_0};R)$ itself corresponding each homology class which is UFG to its dual: the image is a submodule of $A_0$
 or $H^{\ast}(W_{f_0};R)$. We can define the dual of $e_{j^{\prime}}$ and denote it by ${e_{j^{\prime}}}^{\ast} \in H^{\dim S-k}(W_{f^{\prime}};R)$ {\rm (}we may regard that the relation $W_{f_0} \subset W_{f_1} \subset W_{f^{\prime}}$ holds{\rm )}. Let $\{a_{j^{\prime}}\}_{j^{\prime}=1}^{b}$ be a sequence of elements of $R$ each $a_{j^{\prime}}$ of which is represented as $a_{j^{\prime},0} \in \mathbb{Z}$ times the identity element $1$. 
We can define the dual ${c_S}^{\ast}$ so that ${c_S}^{\ast}(c_S)=1$ and that by considering $H_k(S;R)$ as the internal direct sum of the submodule generated by the one element set $\{c_S\}$ and a suitable submodule $C_S$, ${c_S}^{\ast}(C_S)=\{0\}$. We denote the Poincar\'{e} dual to ${c_S}^{\ast}$ in $S$ by ${\rm PD}({c_S}^{\ast})$. Fix a generator of $H^{(n-\dim S)-(n-\dim S)}(S;R)$, isomorphic to $R$: it is regarded as an element of degree $n-\dim S$ in the whole cohomology ring $H^{\ast}(W_{f^{\prime \prime}};R)$ considering the bubbled element. In this situation, the following properties hold.
\begin{enumerate}
\item 
\label{thm:1.2.2.3.0}
 The product of an element in
 ${i_{(f^{\prime},f^{\prime \prime}),S}}^{{\ast}^{\prime}}({{\bar{Q}}_{\dim S-k,W_{f_0},R}}^{\ast})$
 where ${{\bar{Q}}_{\dim S-k,W_{f_0},R}}^{\ast}$ is the module generated by the
 set of all duals of all UFGs in ${\bar{Q}}_{\dim S-k,W_{f_0},R}$
 and an element in ${\rm bub}_{(f^{\prime},f^{\prime \prime}),S_{j_1}}(H^{i_1-(n-\dim S_{j_1})}(S_{j_1};R))$ vanishes.
\item
\label{thm:1.2.2.3.1}
The product of ${i_{(f^{\prime},f^{\prime \prime}),S}}^{{\ast}^{\prime}}({e_{j^{\prime}}}^{\ast})$ and a fixed generator of $H^{(n-\dim S)-(n-\dim S)}(S;R)$, isomorphic to $R$, which is regarded as an element of degree $n-\dim S$ in the whole cohomology ring $H^{\ast}(W_{f^{\prime \prime}};R)$ considering the bubbled element, is $a_{j^{\prime},0}$ times the bubbled element of the dual of ${\rm PD}({c_S}^{\ast}) \in H_{\dim S-k}(S;R)$: it is a class of $H^{n-k}(W_{f^{\prime \prime}};R)$. 

\item
\label{thm:1.2.2.3.2}
 The product of ${i_{(f^{\prime},f^{\prime \prime}),S}}^{{\ast}^{\prime}}({e_{j^{\prime}}}^{\ast})$
 and the dual ${c_S}^{\ast}$  of $c_S \in H_{k}(S;R)$, regarded as an element of degree $n-\dim S+k$ in the whole cohomology ring considering the bubbled element, is $a_{j^{\prime},0}$ times a fixed generator of $H^{n-(n-\dim S)}(S;R)$, isomorphic to $R$, which is regarded as a class of $H^{n}(W_{f^{\prime \prime}};R)$ considering the bubbled element.

\item 
\label{thm:1.2.2.3.3}
For $i \neq 0,\dim S-k$, the product of each element of $${i_{(f^{\prime},f^{\prime \prime}),S}}^{{\ast}^{\prime}}(H^{i}(W_{f^{\prime}};R))$$ and a fixed generator of $H^{(n-\dim S)-(n-\dim S)}(S;R)$, isomorphic to $R$, which is regarded as an element of degree $n-\dim S$ in the whole cohomology ring $H^{\ast}(W_{f^{\prime \prime}};R)$ considering the bubbled element, is zero.

For $i \neq 0,\dim S-k$, the product of each element of $${i_{(f^{\prime},f^{\prime \prime}),S}}^{{\ast}^{\prime}}(H^{i}(W_{f^{\prime}};R))$$ and the dual of $c_S \in H_{k}(S;R)$, regarded as an element of degree $n-\dim S+k$ in the whole cohomology ring considering the bubbled element, is zero.
\end{enumerate}
\end{enumerate}
\end{enumerate}
\end{enumerate}
\end{Thm}
\begin{proof}
The first statement (\ref{thm:1.1}) is almost straightforward by the assumption on the given special generic map, the definition of a connected sum of stable fold maps and the statement and the discussion of the proof of Proposition \ref{prop:3}. We need to explain about the third property. $A_0$ is free and realized as and identified with the cohomology ring of the Reeb space of the special generic map $f_0$ whose coefficient ring is $R$. After the finite iteration of bubbling operations, the original Reeb space is regarded as a subspace of the new Reeb space $W_{f_1}$ by Corollary \ref{cor:1} or as in the proof of Proposition \ref{prop:3}. We define a new cocycle of the new Reeb space $W_{f_1}$ corresponding to an $i$-cocycle ($i>0$) of the original Reeb space $W_{f_0}$ so that the following properties hold: we consider natural triangulations of the Reeb spaces so that $W_{f_0}$ is a subcomplex of $W_{f_1}$. 
\begin{enumerate}
\item At each $i$-chain in the new space $W_{f_1}$, such that no simplex at which the coefficient of the chain is not zero is in the original Reeb space $W_{f_0} \subset W_{f_1}$, the value is $0$.
\item At each $i$-chain regarded as a chain in the original Reeb space by the inclusion, the value is same as the value of the original cocycle at the same chain.
\end{enumerate} 
This canonically gives a monomorphism of $A_0$ into $A$, which is realized as the cohomology ring of $W_{f_1}$ whose coefficient ring is $R$.

We show the statement (\ref{thm:1.2}). 
We consider the case $l=1$ (set $S:=S_1$, $F:=F_1$ and $k:=k_1$).

The generating manifold, which is diffeomorphic to $S$ and we will denote by $S$ samely, of the M(S)-bubbling operation is chosen in the interior of the Reeb space of $W_{f_0}$ and sufficiently close to the boundary of this. This will be explained again with a presentation of the reason in discussing the proof of the second property (\ref{thm:1.2.2}).
The first property （\ref{thm:1.2.1}） is shown by discussions similar to the proof of Proposition \ref{prop:3} (also for the cohomology rings). 
However, we give descriptions via bubbled elements represented by homology classes in Definition \ref{def:4}, {\it bubbled elements} represented by cohomology classes and bubbled spaces in Definition \ref{def:5}.
For a cohomology class $c \in H^{j-(n-\dim S)}(S;R)$, we consider the dual of the class represented by a fiber of the bundle $B((f^{\prime},f^{\prime \prime}),S)$ and the product. Note that $B((f^{\prime},f^{\prime \prime}),S)$ is the total space of an $S^{n-{\dim S}}$-bundle. We consider the product of the two classes and obtain a class in $H^{j}(B((f^{\prime},f^{\prime \prime}),S);R)$ ($c$ is regarded as an element in $H^{j-(n-\dim S)}(B((f^{\prime},f^{\prime \prime}),S);R)$ by regarding this as a canonically obtained class which vanishes at any chain represented as a chain in the fibers of $B((f^{\prime},f^{\prime \prime}),S)$). We can define the {\it bubbled element} ${\rm bub}_{(f^{\prime},f^{\prime \prime}),S}(c) \in H^{j}(W_{f^{\prime \prime}};R)$ by extending this class as the class such that
 at each $j$-chain containing no $j$-simplex at which the coefficient is not zero in $B((f^{\prime},f^{\prime \prime}),S)$, the value is $0$. 

For the inclusion $i_{(f^{\prime},f^{\prime \prime}),S}:W_{f^{\prime}} \rightarrow W_{f^{\prime \prime}}$, we can define a monomorphism ${i_{(f^{\prime},f^{\prime \prime}),S}}^{{\ast}^{\prime}}$ from the cohomology ring $H^{\ast}(W_{f^{\prime}};R)$ into the cohomology ring $H^{\ast}(W_{f^{\prime \prime}};R)$ similarly to the monomorphism from $A_0$ into $A$ (from $H^{\ast}(W_{f_0};R)$ into $H^{\ast}(W_{f_1};R)$) before.

Thus we can define a homomorphism ${\phi}_{(f^{\prime},f^{\prime \prime})}:H^{i}(W_{f^{\prime}};R) \oplus {\oplus}_{j=1}^{l} H^{i-(n-\dim S_j)}(S_j;R) \rightarrow H^{i}(W_{f^{\prime \prime}};R)$ by ${\phi}_{(f^{\prime},f^{\prime \prime}),S}(c_1 \oplus c_2)={i_{(f^{\prime},f^{\prime \prime}),S}}^{{\ast}^{\prime}}(c_1)+{\rm bub}_{(f^{\prime},f^{\prime \prime}),S}(c_2)$
 for $c_1 \in H^{i}(W_{f^{\prime}};R)$ and $c_2 \in H^{i-(n-\dim S)}(S;R)$. By the observation as in the proof of Proposition \ref{prop:3}, this is also an isomorphism. We have remaining statements in the first property （\ref{thm:1.2.1}） immediately.

We discuss the second property (\ref{thm:1.2.2}). The first case (\ref{thm:1.2.2.1}) is shown by a discussion similar to one in the proof of (\ref{thm:1.1}) or (\ref{thm:1.1.3}). $n-\dim S \geq \frac{n}{2}$ holds by the assumption and this or the definition of a bubbled element yields the second case
 (\ref{thm:1.2.2.2}): the product of two bubbled elements of degree larger or equal to $\frac{n}{2}$ is zero.
For the third case (\ref{thm:1.2.2.3}), first we choose the generating manifold $S_{\rm S}$ diffeomorphic to ($\dim S-k$)-dimensional standard sphere in ${\rm Int} W_{f^{\prime}}$ representing the class ${\Sigma}_{j^{\prime}=1}^{b} a_{j^{\prime}} e_{j^{\prime}}$ sufficiently close to the boundary of the original Reeb space $W_{f_0} \subset W_{f_1}$ of the special generic map. We can do this since the dimensions of generating polyhedra for the construction of $f_1$ from $f_0$, which are considered before constructing the connected sum of $f$ and $f_1$, are assumed to be smaller than $n-1<n$ and two relations $n \geq 3$ and $2{\dim} S \leq n$ on dimensions are assumed. The definition of $Q_{\dim S-k,W_{f_0},R} \subset A_0$ is also essential.

The assumptions that $c_S$ in the manifold $S$ is represented by a closed submanifold $F$ whose normal bundle is trivial and that $\dim S \leq \frac{n}{2}$ enable us to obtain an embedding into a small open tubular neighborhood of $S_{\rm S}$, which is regarded as the interior of the total space of a suitable trivial linear bundle, so that the composition of the embedding and the canonical projection to $S_{\rm S}$ is a smooth map and that for the resulting smooth map there exists a preimage of a regular value regarded as $F \subset S$. Now we can demonstrate calculations.

We consider the first case (\ref{thm:1.2.2.3.0}) and the second case (\ref{thm:1.2.2.3.1}). We consider the product of ${i_{(f^{\prime},f^{\prime \prime}),S}}^{{\ast}^{\prime}}({e_{j^{\prime}}}^{\ast})$ and a fixed generator of $H^{(n-\dim S)-(n-\dim S)}(S;R)$, isomorphic to $R$, which is regarded as a class of degree $n-\dim S$ in the whole cohomology ring $H^{\ast}(W_{f^{\prime \prime}};R)$ by considering the bubbled element.

Note also that ${\rm PD}({c_S}^{\ast})$ is regarded as ${\Sigma}_{j^{\prime}=1}^{b} a_{j^{\prime}} e_{j^{\prime}}$ where we consider $S$ as a manifold smoothly embedded in the open tubular neighborhood of $S_{\rm S}$.
We can calculate the value at the tensor product of a cycle representing ${\Sigma}_{j^{\prime}=1}^{b} a_{j^{\prime}} {i_{(f^{\prime},f^{\prime \prime}),S}}_{\ast}(e_{j^{\prime}}) \in {i_{(f^{\prime},f^{\prime \prime}),S}}_{\ast}(H_{\dim S}(W_{f^{\prime}};R))$ (where we consider the monomorphism from a homology group to another homology group induced by the inclusion $i_{(f^{\prime},f^{\prime \prime}),S)}$) and
a cycle representing the bubbled element of a fixed generator of $H_{(n-\dim S)-(n-\dim S)}(S;R)$ and the value is $a_{j^{\prime}}$ (we consider a canonically obtained class of the product space $W_{f^{\prime \prime}} \times W_{f^{\prime \prime}}$: the product is obtained by composing the pull-back via the diagonal map). This completes the calculation for the second case. We can also see that the product of the first case vanishes.

We can calculate the product of the third case (\ref{thm:1.2.2.3.2}) similarly. We can calculate the value at the tensor product of a cycle representing ${\Sigma}_{j^{\prime}=1}^{b} a_{j^{\prime}} {i_{(f^{\prime},f^{\prime \prime}),S}}_{\ast}(e_{j^{\prime}})$ and a cycle representing the bubbled element
 of the class $c_S$ in $S$(, represented by $F$,) before and we have $a_{j^{\prime}}$ (as before we consider a canonically obtained cocycle of the product space $W_{f^{\prime \prime}} \times W_{f^{\prime \prime}}$: the product is obtained by composing the pull-back via the diagonal map). Note also that $c_S$ is a UFG.

We can prove the fourth case (\ref{thm:1.2.2.3.3}) by observing the topologies of the Reeb spaces in this situation easily.

We can show for $l>1$ similarly: rigorous proofs are left to readers.
\end{proof}

\begin{Thm}
\label{thm:2}
In the situation of Theorem \ref{thm:1} {\rm (}\ref{thm:1.2}{\rm )}, we also assume the following conditions.
\begin{enumerate}
\item There exists a UFG $c_0 \in A_0$ of degree $k^{\prime}>0$ which is the dual of a UFG in $Q_{\dim S-k,W_{f_0},R} \subset H_{\dim S-k}(W_{f_0};R)$ and identified with an element of $A_0$ via the isomorphism for the identification of $H^{\ast}(W_{f_0};R)$ with $A_0$ and that $2k^{\prime} \leq n$ holds.
\item $r \in R$ represented as $r_0 \in \mathbb{Z}$ times the identity element $1 \neq 0 \in R$ and an integer $k \geq 0$ satisfying $k+k^{\prime} \leq \frac{n}{2}$ are given.
\end{enumerate}
In this situation, we can construct a fold map $f^{\prime \prime \prime}$ by a normal M{\rm (}S{\rm )}-bubbling operation to the resulting map $f^{\prime \prime}$ in Theorem \ref{thm:1} such that the following properties hold.
\begin{enumerate}
\item
\label{thm:2.1}

 $H^{\ast}(W_{f^{\prime \prime \prime}};R)$ is, as a graded module over $R$, isomorphic to a graded commutative algebra represented as the direct sum of $H^{\ast}(W_{f^{\prime \prime}};R)$ and a suitable module $A^{\prime}$ over $R$: we identify these two isomorphic modules over $R$ in a suitable way.
$H_{i}(W_{f^{\prime \prime \prime}};R)$ is isomorphic to a module over $R$ represented as the direct sum of $H_{i}(W_{f^{\prime \prime}};R)$ and a module isomorphic to the $i$-th module of $A^{\prime}$ over $R$: we identify these two isomorphic modules over $R$ in a suitable way for every $i$.
\item
\label{thm:2.2}
 $A^{\prime}$ is a graded module over $R$ such that for $i=n-k-k^{\prime},n-k^{\prime},n-k,n$, the $i$-th module is $R$ and for $i \neq n-k-k^{\prime},n-k^{\prime},n-k,n$, the $i$-th module is zero.
\item
\label{thm:2.3}
 Products of elements of $H^{\ast}(W_{f^{\prime \prime \prime}};R)$ are as the following lists show.
\begin{enumerate}
\item
\label{thm:2.3.1}
$H^{\ast}(W_{f^{\prime \prime \prime}};R)$ is identified with $H^{\ast}(W_{f^{\prime \prime}};R) \oplus A^{\prime}$ as modules over $R$ in a suitable way as before.
\item
\label{thm:2.3.2}
The product of $(c_1,0)  \in H^{\ast}(W_{f^{\prime \prime}};R) \oplus A^{\prime}$ and $(c_2,0) \in H^{\ast}(W_{f^{\prime \prime}};R) \oplus A^{\prime}$ is $(c_1c_2,0)  \in H^{\ast}(W_{f^{\prime \prime}};R) \oplus A^{\prime}$ where the cohomology ring is identified with $H^{\ast}(W_{f^{\prime \prime \prime}};R)$ as before.
\item
\label{thm:2.3.3}
The product of $(0,c_1)  \in H^{\ast}(W_{f^{\prime \prime}};R) \oplus A^{\prime}$ and $(0,c_2) \in H^{\ast}(W_{f^{\prime \prime}};R) \oplus A^{\prime}$ is $(0,0) \in H^{\ast}(W_{f^{\prime \prime}};R) \oplus A^{\prime}$ where the cohomology ring is identified with $H^{\ast}(W_{f^{\prime \prime \prime}};R)$ as before.
\item
\label{thm:2.3.4}
As a module over $R$, $H^{\ast}(W_{f^{\prime \prime}};R)$ is represented as and in a suitable way identified with the internal direct sum of the free submodule $<c_0>$ generated by the one element set $\{c_0\}$ and a suitable submodule $A^{\prime \prime}$.
\item
\label{thm:2.3.5}
 The product of $(0,c,0)  \in <c_0> \oplus A^{\prime \prime} \oplus A^{\prime}$ and $(0,0,c^{\prime}) \in <c_0> \oplus A^{\prime \prime} \oplus A^{\prime}$ is always $(0,0,0) \in <c_0> \oplus A^{\prime \prime} \oplus A^{\prime}$ where the cohomology ring is identified with $H^{\ast}(W_{f^{\prime \prime}};R) \oplus A^{\prime}$ and $H^{\ast}(W_{f^{\prime \prime \prime}};R)$ as before.
\item 
\label{thm:2.3.6}
The product of $(c_0,0,0)  \in <c_0> \oplus A^{\prime \prime} \oplus A^{\prime}$ in the $k^{\prime}$-th module and $(0,0,1) \in <c_0> \oplus A^{\prime \prime} \oplus A^{\prime}$ in the {\rm (}$n-k-k^{\prime}${\rm )}-th module is $(0,0,r) \in <c_0> \oplus A^{\prime \prime} \oplus A^{\prime}$ in the {\rm (}$n-k${\rm )}-th module where the cohomology ring is identified with $H^{\ast}(W_{f^{\prime \prime}};R) \oplus A^{\prime}$ and $H^{\ast}(W_{f^{\prime \prime \prime}};R)$ as before and where $1$ and $r$ denotes the elements of $R$, isomorphic to the module.
\item 
\label{thm:2.3.7}
The product of $(c_0,0,0)  \in <c_0> \oplus A^{\prime \prime} \oplus A^{\prime}$ in the $k^{\prime}$-th module and $(0,0,1) \in <c_0> \oplus A^{\prime \prime} \oplus A^{\prime}$ in the {\rm (}$n-k^{\prime}${\rm )}-th module is $(0,0,r) \in <c_0> \oplus A^{\prime \prime} \oplus A^{\prime}$ in the $n$-th module where the cohomology ring is identified with $H^{\ast}(W_{f^{\prime \prime}};R) \oplus A^{\prime}$ and $H^{\ast}(W_{f^{\prime \prime \prime}};R)$ as before and where $1$ and $r$ denote the elements of $R$, isomorphic to the module.
\end{enumerate}
\end{enumerate}
\end{Thm}
\begin{proof}
We can show this by in the situation including the proof of Theorem \ref{thm:1}, adding one more submanifold $S_{l+1}:=S^{k^{\prime}} \times S^k$: the class $c_{S_{l+1}}$ is represented by a manifold $\{\ast\} \times S^k \subset S_{l+1}$. $A_0$ is realized as the cohomology ring of the Reeb space of a suitable special
 generic map into ${\mathbb{R}}^n$ such that the restriction to the singular set is an embedding and these graded commutative algebras are identified in a suitable way. The class identified with $c_0$ by this identification is realized by the dual of the class ${c_0}^{\ast}$, represented by a standard sphere of dimension $k^{\prime}$ smoothly embedded in ${\rm Int} W_{f_0}$ and located close to the boundary of $W_{f_0}$. The generating manifold of the ($l+1$)-th normal M(S)-bubbling operation, diffeomorphic to $S_{l+1}$, can be taken in ${\rm Int} W_{f_0}$ and located sufficiently close to the boundary of $W_{f_0}$ so that $S^{k^{\prime}} \times \{{\ast}^{\prime}\} \subset S_{l+1}$ represents the class $r {c_0}^{\ast}=r_0 {c_0}^{\ast}$ as the proof of Theorem \ref{thm:1}.

Investigating the proof of Theorem \ref{thm:1} again, we can check that this completes the proof.
The first two statements are obtained mainly by (the proof of) Theorem \ref{thm:1} (\ref{thm:1.2.1}) and by Proposition \ref{prop:3} where the generating manifold is diffeomorphic to $S^{l+1}$. The third statement on the products are obtained mainly by (the proof of) Theorem \ref{thm:1} (\ref{thm:1.2.2}). For example, the last three lists are owing to (the proof of) (\ref{thm:1.2.2.3}).
\end{proof}

\begin{Ex}
\label{ex:5}
In \cite{kitazawa6}, we studied maps obtained by finite iterations of M(S)-bubbling operations whose generating polyhedra are bouquets of spheres starting from standard GCPS special generic maps from $m$-dimensional manifolds into ${\mathbb{R}}^n$ satisfying $m>n$. 

For example, we consider a suitable example for $n=6$. Let the image of the given standard GCPS special generic map $f^{\prime}$ be diffeomorphic to $S^2 \times D^4$: we assume that $f^{\prime}$ is $f_0$ and we also assume that $f$ is a canonical projection of an unit sphere and that we can regard $f^{\prime \prime}=f^{\prime}=f_1=f_0$ and $A_0=A$ for example in Theorem \ref{thm:1}. In other words, we perform no bubbling operation to $f_0$ to obtain $f_1$, and after that $f^{\prime}$ and $f^{\prime \prime}$, $l=0$ and the connected sum is regarded as an operation preserving the type of the map modulo {\it $C^{\infty}$ equivalence} (this is a fundamental notion in the singularity theory of differentiable maps and forms of fold maps at singular points in the introduction and so on can be represented via the notion: \cite{golubitskyguillemin} explains about the notion well for example). We consider a case where the generating manifold of a normal M(S)-bubbling operation to obtain a new map $f^{\prime \prime \prime}$ from $f^{\prime \prime}$ is diffeomorphic to $S^2 \times S^1$. We set $(k,k^{\prime})=(1,2)$ to apply Theorem \ref{thm:2}.  

The resulting homology group $H_i(W_{f^{\prime \prime \prime}};\mathbb{Z})$ is isomorphic to $\mathbb{Z}$ for $i=0,2,3,4,5,6$ and zero for $i \neq 2,3,4,5,6$. 

Observe the product of a suitable pair of a UFG class in $H^2(W_{f^{\prime \prime}};\mathbb{Z})$ and one in $H^3(W_{f^{\prime \prime}};\mathbb{Z})$. It may not be zero or a UFG.
The resulting Reeb space is also simply-connected.  

Consider another standard GCPS special generic map $f_0$ and pose
 the same assumption as before on the maps $f^{\prime \prime}=f^{\prime}=f_1=f_0$ where we abuse the notation before. We cannot obtain a fold map $f^{\prime \prime \prime \prime}$ such that the cohomology ring $H^{\ast}(W_{f^{\prime \prime \prime \prime}};R)$ is isomorphic to the cohomology ring $H^{\ast}(W_{f^{\prime \prime \prime}};R)$ before by any finite iteration of normal M(S)-bubbling operations whose generating manifolds are standard spheres or points starting from $f^{\prime \prime}$ here.

\end{Ex}

We can generalize this case.

\begin{Thm}
\label{thm:3}
In the situation of Theorem \ref{thm:2} together with Theorem \ref{thm:1}, we assume that $f$ is a canonical projection of an unit sphere, that $f_1=f_0$ holds. Moreover, we assume the following conditions.
\begin{enumerate}
\item In the situation of Theorem \ref{thm:1}, for any pair of elements $c_1,c_2 \in A_0$ which are UFGs such that the sum
 of the degrees are smaller than $n$, the product is also a UFG unless it vanishes.
\item In the situation of Theorem \ref{thm:1} (\ref{thm:1.2}), we also assume that $S_j$ is a standard sphere or a point for any $1 \leq j \leq l$ and for any $1 \leq j \leq l$ such that $n-\dim S_j>0$, the ($n-\dim S_j$)-th module of $A_0$ is zero.
\item In the situation of Theorem \ref{thm:2}, we assume $k>0$.
\end{enumerate}

In this situation, by any finite iteration of M{\rm (}S{\rm )}-bubbling operations such that the generating polyhedra are bouquets of standard spheres or points and that for the dimension $d$ of each of the ingredients of the generating polyhedra here satisfying $d>0$ the ($n-d$)-th module of $A_0$ is zero starting from the resulting map $f^{\prime \prime}$ in Theorems \ref{thm:1} and \ref{thm:2}, we cannot obtain a fold map $f^{\prime \prime \prime \prime}$ such that the cohomology ring $H^{\ast}(W_{f^{\prime \prime \prime \prime}};R)$ is isomorphic to the cohomology ring $H^{\ast}(W_{f^{\prime \prime \prime}};R)$ if $|r_0|>1$ is assumed and in the module $R$ the identity element $1$ is not an element of a finite order in Theorem \ref{thm:2}.
\end{Thm}
\begin{proof}
By the construction, considering normal M(S)-bubbling operations whose generating manifolds are standard spheres satisfying the extra condition on the dimensions or points together with Proposition \ref{prop:3}, the statement of the last part of Theorem \ref{thm:1} (\ref{thm:1.2.1}), the last part of the proof of Theorem \ref{thm:1} or (\ref{thm:1.2.2.3.3}), and so on, we can see that
 for any pair of elements $c_1,c_2 \in H^{\ast}(W_{f^{\prime \prime \prime \prime}};R)$ which are UFGs such that the sum of the degrees are smaller than $n$, the product is also a UFG unless it vanishes (see also \cite{kitazawa6}). This completes the proof.
\end{proof}
As Example \ref{ex:5} explicitly shows, Theorem \ref{thm:3} can give various families of fold maps which we do not obtain in \cite{kitazawa6} and (the cohomology rings of) the Reeb spaces which we do not obtain there for example.

\begin{Cor}
In the situation of Theorem \ref{thm:3}, if for the cohomology ring $H^{\ast}(W_{f^{\prime \prime \prime}};R)$ in Theorem \ref{thm:2}, the $i$-th module for any $0 \leq i \leq n-1$ is free and of rank at most $1$, $|r_0|>1$ is assumed, and in the module $R$ the identity element $1$ is not an element of a finite order, then for any finite iteration of M(S)-bubbling operations whose generating manifolds are standard spheres or points starting from the map $f^{\prime \prime}$, the cohomology ring of the resulting Reeb space cannot be isomorphic to $H^{\ast}(W_{f^{\prime \prime \prime}};R)$.
\end{Cor} 

\begin{Thm}
\label{thm:4.1}
In the situation of Theorem \ref{thm:1} {\rm (}\ref{thm:1.2}{\rm )}, we also assume the following conditions.
\begin{enumerate}
\item There exists a UFG $c_0 \in A_0$ of degree $k^{\prime}>0$ which is the dual of a UFG in $Q_{\dim S-k,W_{f_0},R} \subset H_{\dim S-k}(W_{f_0};R)$ and identified with an element of $A_0$ via the isomorphism for the identification of $H^{\ast}(W_{f_0};R)$ with $A_0$ and that $2k^{\prime} \leq n$ holds.
\item $r \in R$ represented as $r_0 \in \mathbb{Z}$ times the identity element $1 \neq 0 \in R$ and $r^{\prime} \in R$ are given. 
\item $k^{\prime}$ is even and for $k:=k^{\prime}-1$, $k+k^{\prime} \leq \frac{n}{2}$ holds.
\end{enumerate}

In this situation, we can construct a fold map $f^{\prime \prime \prime}$ by a normal M{\rm (}S{\rm )}-bubbling operation to the resulting map $f^{\prime \prime}$ in Theorem \ref{thm:1} such that the following properties hold.
\begin{enumerate}
\item The module $H_{i}(W_{f^{\prime \prime \prime}};R)$ is isomorphic to a module over $R$ represented as the direct sum of $H_{i}(W_{f^{\prime \prime}};R)$ and a module $A_{{\rm h},i}$ over $R$, which we will explain later.
\item The cohomology ring $H^{\ast}(W_{f^{\prime \prime \prime}};R)$ is, as a graded module over $R$, isomorphic to the direct sum of $H^{\ast}(W_{f^{\prime \prime}};R)$ and a suitable module $A_{\rm ch}$ over $R$, which we will explain later: we identify these two isomorphic modules over $R$ in a suitable way.
\end{enumerate}

If $R:=\mathbb{Z}$, then the following facts hold. 
\begin{enumerate}
\item $A_{{\rm h},i}$ is a module over $R$. For $i=n-k-k^{\prime},n$, the module is $R$. For $i=n-k^{\prime}$, the module is $R/2|r_0|R$ and of order $2|r_0|$ if $r_0 \neq 0$ and $R$ if $r_0=0$. For $i=n-k$, the module is zero if $r_0 \neq 0$ and $R$ if $r_0=0$. For $i \neq n-k-k^{\prime},n-k^{\prime},n-k,n$, $A_{{\rm h},i}$ is zero.
\item The $i$-th module of $A_{\rm ch}$ is a graded module over $R$. For $i=n-k-k^{\prime},n$, the $i$-th module is $R$. For $i=n-k$, the $i$-th module is $R/2|r_0|R$ and of order $2|r_0|$ if $r_0 \neq 0$ and $R$ if $r_0=0$. For $i=n-k^{\prime}$, the $i$-th module is zero if $r_0 \neq 0$ and $R$ if $r_0=0$. For $i \neq n-k-k^{\prime},n-k,n-k^{\prime},n$, the $i$-th module is zero.
\item Similar facts as the five lists from {\rm (}\ref{thm:2.3.1}{\rm)} to {\rm (}\ref{thm:2.3.5}{\rm)} in Theorem \ref{thm:2} hold where $A^{\prime}:=A_{\rm ch}$ and $A^{\prime \prime}$ is same as that in the situation of Theorem \ref{thm:2}. Moreover, the following two facts hold for $r_0=0$.
\begin{enumerate}
\item
The product of $(c_0,0,0)  \in <c_0> \oplus A^{\prime \prime} \oplus A^{\prime}$ in the $k^{\prime}$-th module and $(0,0,1) \in <c_0> \oplus A^{\prime \prime} \oplus A^{\prime}$ in the {\rm (}$n-k-k^{\prime}${\rm )}-th module is $(0,0,r^{\prime}) \in <c_0> \oplus A^{\prime \prime} \oplus A^{\prime}$ in the {\rm (}$n-k${\rm )}-th module where the cohomology ring is identified with $H^{\ast}(W_{f^{\prime \prime}};R) \oplus A^{\prime}$ and $H^{\ast}(W_{f^{\prime \prime \prime}};R)$ as before and where $1$ and $r^{\prime}$ denotes the elements of $R$, isomorphic to the module.
\item 
The product of $(c_0,0,0)  \in <c_0> \oplus A^{\prime \prime} \oplus A^{\prime}$ in the $k^{\prime}$-th module and $(0,0,1) \in <c_0> \oplus A^{\prime \prime} \oplus A^{\prime}$ in the {\rm (}$n-k^{\prime}${\rm )}-th module is $(0,0,r^{\prime}) \in <c_0> \oplus A^{\prime \prime} \oplus A^{\prime}$ in the $n$-th module where the cohomology ring is identified with $H^{\ast}(W_{f^{\prime \prime}};R) \oplus A^{\prime}$ and $H^{\ast}(W_{f^{\prime \prime \prime}};R)$ as before and where $1$ and $r^{\prime}$ denote the elements of $R$, isomorphic to the module.
\end{enumerate}
Moreover, these two facts with $r^{\prime}$ being replaced by $0$ hold for $r_0 \neq 0$.
\end{enumerate}
Furthermore, for $k^{\prime}=2,4,8$, we can replace "$2|r_0|$" by "$|r_0|$".
\end{Thm}

We present a sketch of the proof only. For algebraic topological theory and differential topological theory of {\it oriented} linear bundles and their {\it Euler classes} and {\it Euler numbers}, see \cite{milnorstasheff} for example.
\begin{proof}[A sketch of the proof.]
In the proof of Theorem \ref{thm:2}, we set $S_{l+1}$ as the total space of an oriented linear $S^{k}$-bundle over $S^{k^{\prime}}$ whose Euler number is $2r_0$ instead. For $k^{\prime}=2,4,8$, we set $S_{l+1}$ as the total space of an oriented linear $S^{k}$-bundle over $S^{k^{\prime}}$ whose Euler number is $r_0$.
\end{proof}

\begin{Thm}
\label{thm:4.2}
In the situation of Theorem \ref{thm:1} {\rm (}\ref{thm:1.2}{\rm )}, let $p \geq 1$ be an integer and $R:=\mathbb{Z}/2p\mathbb{Z}$. Let $r_{p} \in R$.
\begin{enumerate}
\item There exists a UFG $c_0 \in A_0$ of degree $k^{\prime}>0$ which is the dual of a UFG in $Q_{\dim S-k,W_{f_0},R} \subset H_{\dim S-k}(W_{f_0};R)$ and identified with an element of $A_0$ via the isomorphism for the identification of $H^{\ast}(W_{f_0};R)$ with $A_0$ and that $2k^{\prime} \leq n$ holds.
\item $k^{\prime}$ is even and for $k:=k^{\prime}-1$, $k+k^{\prime} \leq \frac{n}{2}$ holds.
\end{enumerate}

In this situation, we can construct a fold map $f^{\prime \prime \prime}$ by a normal M{\rm (}S{\rm )}-bubbling operation to the resulting map $f^{\prime \prime}$ in Theorem \ref{thm:1} such that the following properties hold.
\begin{enumerate}
\item The module $H_{i}(W_{f^{\prime \prime \prime}};R)$ is isomorphic to a module over $R$ represented as the direct sum of $H_{i}(W_{f^{\prime \prime}};R)$ and a module $A_{{\rm h},i}$ over $R$, which we will explain later.
\item The cohomology ring $H^{\ast}(W_{f^{\prime \prime \prime}};R)$ is, as a graded module over $R$, isomorphic to the direct sum of $H^{\ast}(W_{f^{\prime \prime}};R)$ and a suitable module $A_{\rm ch}$ over $R$, which we will explain later: we identify these two isomorphic modules over $R$ in a suitable way. 
\item For $i=n-k-k^{\prime},n$, $A_{{\rm h},i}$ is $R$. For $i=n-k^{\prime},n-k$, $A_{{\rm h},i}$ is $R$. For $i \neq n-k-k^{\prime},n-k^{\prime},n-k,n$, $A_{{\rm h},i}$ is zero. The $i$-th module of $A_{\rm ch}$ is isomorphic to this graded module over $R$.
\item
Similar facts as the seven lists from {\rm (}\ref{thm:2.3.1}{\rm)} to {\rm (}\ref{thm:2.3.7}{\rm)} in Theorem \ref{thm:2} hold where $A^{\prime}:=A_{\rm ch}$, $A^{\prime \prime}$ is same as that in the situation of Theorem \ref{thm:2} and $r$'s in {\rm (}\ref{thm:2.3.6}{\rm)} and {\rm (}\ref{thm:2.3.7}{\rm)} are replaced by $r_p$'s.
\item $f^{\prime \prime \prime}$ can be obtained in Theorem \ref{thm:4.1} as an explicit example for the case $r_0=2p \neq 0$ in the situation of Theorem \ref{thm:4.1}.
\end{enumerate}

Furthermore, for $k^{\prime}=2,4,8$ and $p>1$, we can replace $R:=\mathbb{Z}/2p\mathbb{Z}$ by $R:=\mathbb{Z}/p\mathbb{Z}$ and $r_0=2p$ by $r_0=p$.
\end{Thm}
\begin{proof}[A sketch of the proof.]
We can prove as Theorem \ref{thm:4.1}.
In the proof of Theorem \ref{thm:2}, we set $S_{l+1}$ as the total space of an oriented linear $S^{k}$-bundle over $S^{k^{\prime}}$ whose Euler number is $2p$ instead. For $k^{\prime}=2,4,8$ and $p>1$, we set $S_{l+1}$ as the total space of an oriented linear $S^{k}$-bundle over $S^{k^{\prime}}$ whose Euler number is $p$.
\end{proof}
\begin{Cor}
Theorem \ref{thm:4.2} with Theorem \ref{thm:4.1} presents pairs of Reeb spaces for each of which the following properties hold.
\begin{enumerate}
\item The cohomology rings of both of the pair are isomorphic under the condition that the coefficient ring is $\mathbb{Z}$.  
\item The cohomology rings of the two Reeb spaces are not isomorphic under the condition that the coefficient ring is $\mathbb{Z}/p\mathbb{Z}$ or $\mathbb{Z}/2p\mathbb{Z}$. 
\end{enumerate} 
\end{Cor}
Last we present a result which may be fundamental and useful in constructing new stable fold maps from known ones. It resembles the last theorem of \cite{kitazawa5} for statements on homology groups.
\begin{Thm}
\label{thm:5}
Let $f$ be a stable fold map from a closed and connected manifold of dimension $m$ into ${\mathbb{R}}^n$ satisfying $m>n \geq 1$.

Let $R$ be a PID. For any integer $0 \leq j \leq n$, let a finitely generated module $G_j$ over $R$ be given such that $G_0$ is trivial and that $G_n$ is not zero. 

Let a stable special generic map $f_0$ be given. Assume that by
 performing a finite iteration of S-bubbling operations starting from this special generic map whose generating polyhedra are in the interior of the Reeb space of this original special generic map disjointly so that at each step the remaining polyhedra are apart from the bubbled spaces and by considering a connected sum of $f$ and the resulting map $f_1$ just before, we can obtain a fold map $f^{\prime}$ such that the module $H_j(W_{{f}^{\prime}};R)$ is isomorphic to the module $H_j(W_f;R) \oplus G_j$: note that $G_n$ is free by virtue of Proposition \ref{prop:3}. Let $H$ be
 a non-trivial submodule of $G_n$ {\rm (}and free{\rm )}. 

Last, assume also that the connected components of preimages born by bubbles in the finite iteration of the operations to get $f_1$ from $f_0$ are all diffeomorphic to a given homotopy sphere.

In this situation, the following two facts hold.
\begin{enumerate}
\item Starting from the given map $f$ and the special generic map $f_0$ above similarly, we can obtain a fold map $f^{\prime \prime}$ by a finite iteration of S-bubbling operations starting from this special generic map $f_0$ whose generating polyhedra are in the interior of the Reeb space of this original special generic map and after that by considering a connected sum of $f$ and the new resulting map. 
\item Furthermore, we can construct $f^{\prime \prime}$ satisfying the following properties.
\begin{enumerate}
\item The module $H_j(W_{{f}^{\prime \prime}};R)$ is isomorphic to the module $H_j(W_f;R) \oplus G_j$ for $0 \leq j \leq n-1$.
\item The module $H_n(W_{{f}^{\prime \prime}};R)$ is isomorphic to the module $H_n(W_f;R) \oplus H \subset H_n(W_f;R) \oplus G_n$.
\item The product on the ring $H^{\ast}(W_{{f}^{\prime \prime}};R)$ is defined by composing
 the product on the ring $H^{\ast}(W_{{f}^{\prime}};R)$ with a natural projection to $H^{\ast}(W_{{f}^{\prime \prime}};R)$. For $0 \leq j \leq n$, only in the case $j=n$, $H^j(W_{{f}^{\prime \prime}};R)$ is not isomorphic to $H^j(W_{{f}^{\prime \prime}};R)$ and they are regarded as $H^n(W_f;R) \oplus H \subset H^n(W_f;R) \oplus G_n$, respectively. The natural projection is regarded as the identity map on the $j$-th module for $0 \leq j<n$ considering natural identifications of the $j$-th modules for $0 \leq j \leq n$. 
\end{enumerate} 
\end{enumerate}
\end{Thm}
\begin{proof}
The main ingredient of the proof is how we take the generating polyhedra. If the difference between the ranks of $G_n$ and $H$ is $l>0$, then we can choose $l+1$ of the original generating polyhedra in the interior of the Reeb space of the original special generic map $f_0$ and consider the bouquet without changing the isotopy classes of the original generating polyhedra in the (interior of the) Reeb space (of the special generic map). 

We can consider a new finite iteration of bubbling operations respecting this situation starting
 from the given special generic map since the connected components of preimages born by bubbles in the finite iteration of the operations to get $f_1$ from $f_0$ are all diffeomorphic to a given homotopy sphere and after that we can perform a connected sum of $f$ and this new resulting map. By virtue of some discussions in the proofs of Proposition \ref{prop:3} and Theorem \ref{thm:1} and the topologies of the Reeb spaces, we can see that the resulting stable fold map is a desired one.  
\end{proof}
We can apply Theorem \ref{thm:5} to some maps constructed by using Theorems \ref{thm:1}, \ref{thm:2}, \ref{thm:3}, \ref{thm:4.1}, \ref{thm:4.2} and so on. For example, we can obtain fold maps such that the cohomology rings of the Reeb spaces are not isomorphic to ones obtained in \cite{kitazawa6} again.
We can also apply Proposition \ref{prop:2} to some maps constructed via applications of these theorems.





\end{document}